\documentclass[10pt]{amsart}
\usepackage{latexsym,amsmath,amssymb,graphicx}
\usepackage{float}
\usepackage[bf, small]{caption}
\usepackage[all,web]{xy}

\usepackage{hyperref}

\DeclareFontFamily{OT1}{rsfs10}{}
\DeclareFontShape{OT1}{rsfs10}{m}{n}{ <-> rsfs10 }{}
\DeclareMathAlphabet{\mathscript}{OT1}{rsfs10}{m}{n}

\DeclareMathOperator{\Hom}{Hom}     
\DeclareMathOperator{\Pic}{Pic}     
\DeclareMathOperator{\Cl}{Cl}       
\DeclareMathOperator{\rk}{rk}       
\DeclareMathOperator{\Mov}{Mov}     
\DeclareMathOperator{\Nef}{Nef}     
\DeclareMathOperator{\Eff}{Eff}     
\DeclareMathOperator{\Relint}{Relint}  

\def \b{\beta }

\def \l{\lambda }
\def \s{\sigma }

\def \Si{\Sigma }
\def \g{\gamma}

\def \q{\mathbf{q}}
\def \pp{\mathbf{p}}
\def \ss{\mathbf{s}}

\def \v{\mathbf{v}}
\def \n{\mathbf{n}}

\def \1{\mathbf{1}}
\def \0{\mathbf{0}}

\def\P{{\mathbb{P}}}

\def\p2{\mathbb{P}^2}
\def\p3{\mathbb{P}^3}
\def\p4{\mathbb{P}^4}
\def\r{\mathbf{r}}

\def\rk{\operatorname{rk}}

\def\Z{\mathbb{Z}}

\def\C{\mathbb{C}}
\def\R{\mathbb{R}}

\def\Q{\mathbb{Q}}
\def\N{\mathbb{N}}

\def\B{\mathcal{B}}

\def\SF{\mathcal{SF}}

\def\gkz{\mathcal{Q}}

\newcommand{\halfline}{\vskip6pt}
\newcommand{\ka}{K\"{a}hler }

\theoremstyle{plain}
\newtheorem{theorem}{Theorem}[section]

\newtheorem{thm-def}[theorem]{Theorem--Definition}

\newtheorem{lemma}[theorem]{Lemma}

\newtheorem*{a-proposition}{Proposition}

\theoremstyle{remark}
\newtheorem{remark}[theorem]{Remark}

\newtheorem{example}[theorem]{Example}

\theoremstyle{definition}

\newtheorem*{step I}{Step I}
\newtheorem*{step II}{Step II}
\newtheorem*{step III}{Step III}
\newtheorem*{step IV}{Step IV}

\newtheorem*{acknowledgements}{Acknowledgements}

\title[A $\Q$--F.C. toric variety with Picard number 2 is projective]{A $\Q$--factorial complete toric variety\\ with Picard number 2 is projective}

\author[M. Rossi and L.Terracini]{Michele Rossi and Lea Terracini}

\date{\today}

\address{Dipartimento di Matematica, Universit\`a di Torino,
via Carlo Alberto 10, 10123 Torino} \email{michele.rossi@unito.it,
lea.terracini@unito.it}

\thanks{The authors were partially supported by the MIUR-PRIN 2010-11 Research Funds ``Geometria delle Variet\`{a} Algebriche''. The first author is also supported by the I.N.D.A.M. as a member of the G.N.S.A.G.A.}

\begin{document}

\begin{abstract}
This paper is devoted to settle two still open problems, connected with the existence of ample and nef divisors on a $\Q$--factorial complete toric variety. The first problem is about the existence of ample divisors when the Picard number is 2: we give a positive answer to this question, by studying the secondary fan by means of $\Z$--linear Gale duality. The second problem is about the minimum value of the Picard number allowing the vanishing  of the $\Nef$ cone: we present a 3--dimensional example showing that this value cannot be greater then 3, which, under the previous result, is also the minimum value guaranteeing the existence of non--projective examples.
\end{abstract}

\maketitle

\tableofcontents

\section*{Introduction}

It is classically well known that a complete toric variety of dimension $\leq 2$ is projective \cite[\S\,8, Prop.\,8.1]{MO}. For higher dimension this  fact is not true, as shown by lot of counterexamples, the first of which was given by M.\,Demazure \cite{Demazure}.  For smooth complete toric varieties, P.\,Kleinschmidt and B.\,Sturmfels \cite{Kleinschmidt-Sturmfels} proved that for Picard number (in the following also called the \emph{rank}) $r\leq 3$ they are projective in every dimension. This result cannot be extended to higher values of the rank, as shown by a famous example given in Oda's book \cite[p.84]{Oda}: this is a smooth complete 3--dimensional toric variety $X$ of rank $r=4$, such that $\dim(\Nef(X))=2$ inside $\Cl(X)\otimes\R\cong\R^4$: therefore $X$ admits non--trivial numerically effective classes (among which the anti--canonical one) but does not admit any ample class.

\noindent When dropping the smoothness hypothesis, Kleinschmidt--Sturmfels bound does no longer hold even for $\Q$--factorial singularities: a counterexample has been given by F.\,Berchtold and J.\,Hausen \cite[Ex.\,10.2]{Berchtold-Hausen} who produced a $\Q$--factorial complete 3--dimensional toric variety $X$ of rank $r=3$ such that $\Nef(X)$ is a 1--dimensional cone generated by the anti--canonical class. This example is essentially produced by suppressing a fan generator in the Oda's example (see Remark \ref{rem:}).

\noindent Nevertheless, a $\Q$--factorial complete toric variety of rank $r=1$ is always projective as a quotient of a weighed projective space (so called a \emph{fake} WPS) \cite{Conrads}, \cite{BC}. It is then natural to ask what happens for $r=2$: surprisingly, the current literature does not seem to give any answer to this question, at least as far as we know. The first and main result we are going to present in this paper is a extension of the Kleinschmidt--Sturmfels result to $\Q$--factorial setup, by decreasing the bound on the rank $r$, namely

\halfline
\noindent {\bf Theorem \ref{teo:erredueproj}} \emph{Every $\Q$--factorial complete toric variety of Picard number $r\leq 2$ is projective.}

\halfline \noindent Our proof is essentially obtained by generalizing a Berchtold--Hausen's argument proving the last part of \cite[Prop.\,10.1]{Berchtold-Hausen}, by means of $\Z$--linear Gale duality, as de\-ve\-lo\-ped in \cite{RT-LA&GD} and \cite{RT-QUOT}. In fact $\Z$--linear Gale duality allows us to observe the general fact, summarized by Lemma~\ref{lem:DUEparte}, about the mutual position of pairs of cones Gale dually associated with pairs of maximal cones of the fan admitting a common facet. When the rank is 2, Lemma~\ref{lem:DUEparte} gives rise to strong consequences proving Theorem \ref{teo:erredueproj}. This proof turns out to be very easy and we believe this is a further reason of interest.

\halfline
A second result of the present paper is proposed in the last section \S\,\ref{sez:esempio}, by exhi\-bi\-ting an example of $\Q$--factorial complete and non--projective 3--dimensional toric variety $X$ with rank $r=3$ and \emph{not admitting any non--trivial numerically effective divisor} i.e. such that $\Nef(X)=0$. More precisely, Theorem \ref{teo:erredueproj} implies that
\begin{equation}\label{nef0Q}
\Nef(X)=0\,\Rightarrow\,r\geq 3\,.
\end{equation}
Example given in \S\,\ref{sez:esempio} then shows that condition (\ref{nef0Q}) is sharp. Actually this is not a new example. In fact O.\,Fujino and S.\,Payne proved that $\Nef(X)=0\,\Rightarrow\,r\geq 5$, for a smooth and complete toric 3--fold $X$ \cite{FP}: their Example 1 is obtained by blowing up a complete toric variety which turns out to be isomorphic to the one given in \S\,\ref{sez:esempio}. This example jointly with Theorem \ref{teo:erredueproj} allows us to conclude the sharpness of condition (\ref{nef0Q}), which is a new result.
 Let us finally notice that this example can be produced by \emph{breaking the symmetry} of the above mentioned Berchtold--Hausen example, just \emph{deforming} one generator of the effective cone $\Eff(X)$ inside $\Cl(X)\otimes\R$. For more details see Remark \ref{rem:}.

\section{A general Lemma for arbitrary Picard number}
In the present paper we deal with $\Q$--factorial complete toric varieties associated with simplicial and complete fans. For preliminaries and used notation on toric varieties we refer the reader to \cite[\S~1.1]{RT-LA&GD} and \cite[\S 1]{RT-QUOT}. We will also apply $\Z$--linear Gale duality as developed in \cite[\S~3]{RT-LA&GD}.
Each time the needed nomenclature will be recalled either directly by giving the necessary definition or by reporting the precise reference.

Let $X(\Si)$ be a $\Q$--factorial complete toric variety of dimension $n$ and Picard number $r:=\rk(\Pic(X))$ (in the following also called the \emph{rank of $X$}, $r=\rk(X)$). Then $\Si$ is a rational, simplicial and complete fan in $N_{\R}:=N\otimes\R$, where $N$ is the dual of the group of characters of the acting torus $T\cong(\C^*)^n$. Its 1--skeleton is given by $\Sigma(1)=\{\langle\v_1\rangle,\ldots,\langle\v_{n+r}\rangle\}$, where $\v_i\in N$ is a generator of the monoid associated with the  corresponding ray $\langle\v_i\rangle$. The $n\times (n+r)$ integer matrix $V=(\v_1,\ldots,\v_{n+r})$ will be called \emph{a fan matrix} of $X$ and it turns out to be a reduced $F$--matrix, in the sense of \cite[Def.~3.10]{RT-LA&GD} and \cite[Def.~1.2]{RT-QUOT}.

In the following, given a $d\times m$ integer matrix $A\in\mathbf{M}(d,m;\Z)$ and a subset $I\subseteq\{1,\ldots,m\}$, $A_I$ will denote the submatrix of $A$ given by the columns indexed by $I$, and $A^I$ will denote the submatrix of $A$ whose columns are indexed by the complementary subset $\{1,\ldots,m\}\backslash I$.

A maximal cone $\s\in\Si(n)$ can be identified with a maximal rank $n\times n$ submatrix $V_J=V^I$ of the fan matrix $V$, assigned by all the columns of $V$ generating some ray in the 1-skeleton $\s(1)$ of $\s$: in particular $J=\{j_1,\ldots,j_n\}\subset\{1,\ldots,n+r\}$ denotes the position of those columns inside $V$, while $I=\{1,\ldots,n+r\}\setminus J$ is the complementary subset. We will also write $\s=\s_J=\langle V_J\rangle=\langle V^I\rangle=\s^I$.

Let $Q=(\q_1,\ldots,\q_{n+r})$ be a maximal rank $r\times (n+r)$ \emph{Gale dual} matrix of $V$ i.e. such that $Q\cdot V^T=0$ \cite[\S~3.1]{RT-LA&GD}, where $V^T$ denotes the transposed matrix of $V$. Define the following set of \emph{dual} cones inside $F^r_{\R}=\Cl(X)\otimes\R$:
\begin{equation*}
  \B_{\Si}:=\{\langle Q_I\rangle\subseteq F^r_{\R}\ |\ \langle V^I\rangle\in\Si(n)\}\,
\end{equation*}
where $\langle Q_I\rangle$ is the cone generated by the columns of the submatrix $Q_I$.
$\B_{\Si}$ turns out to be a  \emph{bunch of cones} in the sense of \cite{Berchtold-Hausen} (see also \cite[p.~738]{CLS}). In particular

\begin{theorem}\label{thm:camera} \cite[Prop.15.2.1(c)]{CLS} In the above notation, the \ka cone of the $\Q$--factorial complete toric variety $X(\Si)$ is given by
$$\Nef(X)=\g:=\bigcap_{\b\in\B_{\Si}}\b\,.$$
In particular $X$ is projective if and only $\dim(\g)=r$.

\end{theorem}

Let us set
\begin{equation*}
\mathcal{I}_{\Si}=\{I\subset\{1,\ldots,n+r\}\ |\ \langle V^I\rangle\in \Si(n)\}=\{I\subset\{1,\ldots,n+r\}\ | \ \langle Q_I\rangle\in\B_{\Si}\}\,.
\end{equation*}
Then we get the following:

\begin{lemma}\label{lem:DUEparte} With the notation introduced above:
\begin{itemize}
\item[a)] if $I\in\mathcal{I}_\Sigma$ then for every $j\not\in I$ there exists a unique $k\in I$ such that
\begin{equation}
\label{eq:Iprime}
I'=(I\setminus \{k\})\cup \{j\} \in \mathcal{I}_{\Si}
\end{equation}
\item[b)] let $I, I'$ be as in part a): then the vectors $\mathbf{q}_j$ and $\mathbf{q}_k$ lie on the same side with respect to the hyperplane $H\subset F_{\R}$ generated by $\{\mathbf{q}_{i} \ |\ i\in I\setminus \{k\}\}$.
\end{itemize}
\end{lemma}

\begin{proof} a) This is the dual assertion of the fact that  every facet of $\Sigma$ belongs to exactly two maximal cones of $\Sigma$. Since $j\not \in I$, $\mathbf{v}_j \in \langle V^I\rangle $. Then the facet $\tau$ of $\langle V^I\rangle $ opposite to $\mathbf{v}_j$ belongs to exactly one other maximal cone $ \langle V^{I'}\rangle $. If $\mathbf{v}_{k}$ is the vector opposite to $\tau$ in $ \langle V^{I'}\rangle $ then (\ref{eq:Iprime}) holds.\\

b) Let $L\subset N_{\R}$ be the hyperplane supporting the common face $\tau=\langle V^I\rangle\cap \langle V^{I'}\rangle$: it is generated by vectors $\v_t\,,\,t\not\in I\cup\{j\}$. Since $\mathbf{v}_j$ and
$\mathbf{v}_k$ lie on opposite sides of $L$, the cone generated by $\mathbf{v}_j$ and $\mathbf{v}_k$ intersects $L$ giving a relation
\begin{equation}\label{relazione}
  \lambda_j\v_j+\lambda_k\v_k=\sum_{t\not\in I\cup\{j\}} \lambda_{t}\mathbf{v}_{t}
\end{equation}
such that $\lambda_{j},\lambda_k >0$. Therefore in the sublattice $\mathcal{L}_r(Q)\subseteq\Z^{n+r}$, spanned by the rows of $Q$, there is a vector having positive entries at places $j$ and $k$ and $0$ at places in $I\setminus \{k\}$. It corresponds to a vector $$\n\in\R^r=\Hom(\Cl(X),\R))\,:\ \forall\,i\in I\setminus\{k\}\ \n\cdot \mathbf{q}_i=0\ \text{and}\ \mathbf{n}\cdot \mathbf{q}_j>0\,,\,\mathbf{n}\cdot \mathbf{q}_k>0$$
which is clearly a normal vector to the hyperplane $H\subset F^r_{\R}$.
\end{proof}

\begin{example} To better understand the argument proving Lemma \ref{lem:DUEparte} it may be of some help considering the easy example given by $X=\P^1\times\P^1$. A fan matrix $V$ of $X$ and a Gale dual $Q$ of $V$ are the following
$$V=(\v_1,\v_2,\v_3,\v_4)=\left(
                            \begin{array}{cccc}
                              1 & -1 & 0 & 0 \\
                              0 & 0 & 1 & -1 \\
                            \end{array}
                          \right)\ \Rightarrow\ Q=\left(
                                                    \begin{array}{cccc}
                                                      1 & 1 & 0 & 0 \\
                                                      0 & 0 & 1 & 1 \\
                                                    \end{array}
                                                  \right)\ .
$$
In this case $\mathcal{I}_{\Si}=\{\{2,4\},\{1,4\},\{1,3\},\{2,3\}\}$ and part a) is clear. Let us now consider the two maximal cones $\langle\v_1,\v_3\rangle=\left\langle V^{\{2,4\}}\right\rangle$ and $\langle\v_2,\v_3\rangle=\left\langle V^{\{1,4\}}\right\rangle$, whose common facet is given by $\tau=\langle\v_3\rangle$. In the notation of Lemma~\ref{lem:DUEparte} one has
$I=\{2,4\},\ I'=\{1,4\},\ j=1,\ k=2$: the argument proving part b) gives $t=3$ and relation (\ref{relazione}) is given by
$$1\cdot\v_1+1\cdot\v_2=0\cdot\v_3\quad\Rightarrow\quad \l_1=\l_2=1>0\ .$$
The associated vector in $\mathcal{L}_r(Q)$ turns out to be the first row of $Q$, hence giving $\n=\q_1=\q_2=(1\ 0)^T$ lying on the same side with respect to the hyperplane (i.e. the line) $H$ spanned by the fourth column $\q_4=(0\ 1)^T$ of $Q$.
\end{example}
\begin{figure}
\begin{center}
\includegraphics[width=9truecm]{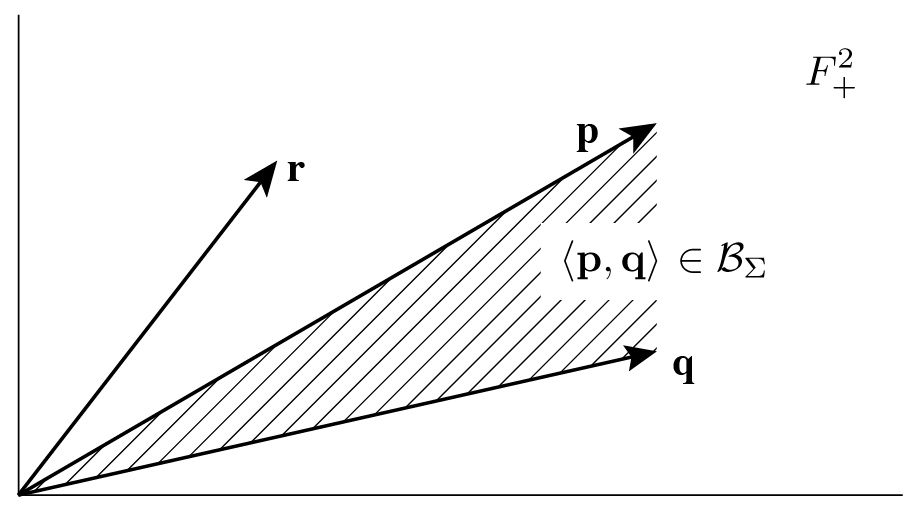}
\caption{\label{Fig1}Cones of the bunch $\B_{\Si}$ as in Lemma \ref{lem:pqr}.b)}
\end{center}
\end{figure}

\section{The case of Picard number $r=2$}

Given a simplicial and complete fan $\Si$, if $r=2$ then $|I|=2$ for every $I\in\mathcal{I}_{\Si}$. The following is a direct application of Lemma \ref{lem:DUEparte}.

\begin{lemma}\label{lem:pqr} Let $\mathbf{p},\mathbf{q},\mathbf{r}$ be three distinct vectors giving as many columns of $Q$, such that $\langle \mathbf{p},\mathbf{q}\rangle \in\B_{\Si}$.
\begin{itemize}
\item[a)] exactly one between $\langle \mathbf{p},\mathbf{r}\rangle$ and  $\langle \mathbf{q},\mathbf{r}\rangle$ is in $\B_{\Si}$.
\item[b)] if $\mathbf{p}\in \langle \mathbf{q},\mathbf{r}\rangle$ then $\langle \mathbf{q},\mathbf{r}\rangle\in \B_{\Si}$ and $\langle \mathbf{p},\mathbf{r}\rangle\not\in \B_{\Si}$.
\end{itemize}
\end{lemma}
\begin{proof}
a) This is a direct consequence of Lemma \ref{lem:DUEparte}.a) by setting $Q_I=\langle \pp,\q\rangle$ and $\q_j=\r$.

b) If $\pp\in\langle \q,\r\rangle$ then $\mathbf{r}\notin \langle \mathbf{p},\mathbf{q}\rangle$ (see Fig. \ref{Fig1}). By a) exactly one between $\langle \mathbf{p},\mathbf{r}\rangle$ and $\langle \mathbf{q},\mathbf{r}\rangle$ belongs to $\B_{\Si}$ and either $Q_{I'}=\langle\pp,\r\rangle$ or $Q_{I'}=\langle\q,\r\rangle$. By Lemma \ref{lem:DUEparte}.b) the former cannot occur, meaning that $Q_{I'}=\langle\q,\r\rangle\in\B_{\Si}$.
\end{proof}

We are now in a position to prove the main result of the present paper.

\begin{theorem}
\label{teo:erredueproj}
Every $\Q$-factorial  complete toric variety with Picard number $r=2$ is projective.
\end{theorem}
\begin{proof} Consider a cone $\langle \pp,\q\rangle \in \B_{\Si}$. If there exists a further column $\r$ of $Q$ such that $\r\in\Relint(\langle\pp,\q\rangle)$ then Lemma \ref{lem:pqr}.a) ensures that either $\langle \pp,\r\rangle$ or $\langle\q,\r\rangle$ belongs to $\B_{\Si}$. Then, by repeatedly applying Lemma \ref{lem:pqr}.a), we get a \emph{minimal} cone $\g\in\B_{\Si}$, in the sense that there are no columns of $Q$ in $\Relint(\g)$.

\noindent We can then assume $\gamma=\langle \mathbf{p},\mathbf{q}\rangle \in\B_{\Si}$ is a minimal chamber: we claim that $\gamma$ is contained in every cone $\b\in \B_{\Si}$.

\noindent In fact, if one of $\mathbf{p},\mathbf{q}$ is a ray of $\b$  then, up to renaming $\pp$ and $\q$, we can assume $\b=\langle\q,\r\rangle$. Lemma \ref{lem:DUEparte}.b) guarantees that either $\pp\in\b$ and we are done or $\r\in\g$, against the minimality of $\g$. Then we can assume $\b=\langle \mathbf{r},\mathbf{s}\rangle$, with $\{\mathbf{p},\mathbf{q}\}\not=\{\mathbf{r},\mathbf{s}\}$. Assume $\gamma\not\subset \b$. By the minimality of $\gamma$, the cone generated by $\mathbf{p},\mathbf{q},\mathbf{r},\mathbf{s}$ has a ray in $\{\mathbf{p},\mathbf{q}\}$ and a ray in $\{\mathbf{r},\mathbf{s}\}$. We may then assume $\mathbf{p},\mathbf{s}\in \langle \mathbf{q},\mathbf{r}\rangle$, as in Fig. \ref{Fig2}. Consider the three vectors $\q,\pp,\ss$: since $\g=\langle\pp,\q\rangle\in\B_{\Si}$ and $\pp\in\langle\q,\ss\rangle$, Lemma \ref{lem:pqr}.b) ensures that $\langle\q,\ss\rangle\in\B_{\Si}$. Finally consider the three vectors $\q,\ss,\r$: since both $\langle \q,\ss\rangle$ and $\b=\langle\r,\ss\rangle$ belong to
$\B_{\Si}$ we get a contradiction with Lemma \ref{lem:DUEparte}.b).

\noindent Then, recalling Theorem \ref{thm:camera},
$\Nef(X(\Si))=\bigcap_{\b\in\B_{\Si}}\b=\g $ is a 2--dimensional cone, meaning that $X$ is projective.
\end{proof}

\begin{figure}
\begin{center}
\includegraphics[width=9truecm]{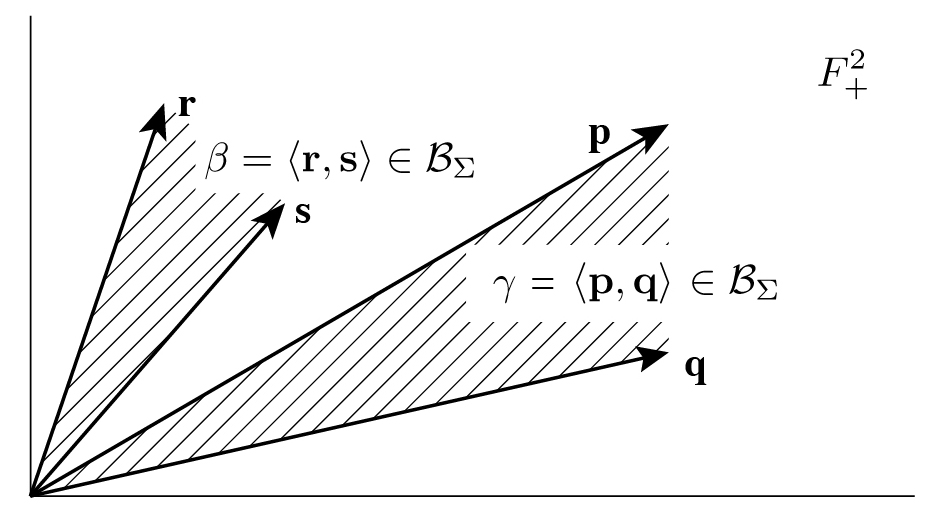}
\caption{\label{Fig2}Cones of the bunch $\B_{\Si}$ as in Theorem \ref{teo:erredueproj}}
\end{center}
\end{figure}

\section{A 3--dimensional counterexample with $r=3$}\label{sez:esempio}

In this section we are going to exhibit an example of a 3--dimensional $\Q$--factorial complete toric variety $X$ of Picard number $r=3$ not admitting \emph{any non--trivial numerically effective divisor} i.e. such that $\Nef(X)=0$. Therefore such an example is significant for at least two reasons:
\begin{itemize}
\item it gives a sharp counterexample to the question if Theorem \ref{teo:erredueproj} is holding for higher values of $r$: in this sense it does not say nothing new with respect to the Berchtold--Hausen example \cite[Ex.\,10.2]{Berchtold-Hausen};
\item for smooth complete toric $3$--folds O.\,Fujino and S.\,Payne \cite{FP} proved that if $\Nef(X)=0$ then $r\geq 5$\,: their Example 1 is produced by performing a double divisorial blow up of a complete toric variety which turns out to be isomorphic to the one we are going to present, showing that this last inequality does no longer hold when the smoothness hypothesis is replaced by the $\Q$--factorial one, actually giving the sharpness of condition (\ref{nef0Q}) in the introduction.
\end{itemize}
Let us start by considering the following positive $W$--matrix \cite[Def.~3.9]{RT-LA&GD}, \cite[Def.~1.4]{RT-QUOT} in row echelon form
\begin{equation*}
Q:=\left(
     \begin{array}{cccccc}
       1&1&0&0&1&0\\
      0&1&1&1&0&0\\
       0&0&0&2&1&1 \\
     \end{array}
   \right)
\end{equation*}
The cone $\gkz=\langle Q\rangle$ spanned by its columns gives the effective cone $\Eff(X)$ of a $\Q$--factorial complete 3--dimensional toric variety $X$ of Picard number $r=3$ whose fan matrix is given by a Gale dual $F$--matrix $V$ of $Q$ \cite[Def.~3.10]{RT-LA&GD}, \cite[Def.~1.2]{RT-QUOT} e.g.
\begin{equation*}
 V=\left(
  \begin{array}{cccccc}
    1&0&0&0&-1&1 \\
    0&1&0&-1&-1&3 \\
    0&0&1&-1&0&2 \\
  \end{array}
\right)\,.
\end{equation*}
This fan matrix can support 8 different rational, simplicial and complete fans whose 1--skeletons coincide with the rays spanned by all the columns of $V$, namely $\SF(V)= \{\Si_i\,|\ 1\leq i\leq 8\}$\ \cite[Def.~1.3]{RT-LA&GD} with
\begin{eqnarray*}
  \Si_1 &=& \{\langle3, 4, 5\rangle, \langle2, 4, 5\rangle, \langle2, 3, 5\rangle, \langle1, 3, 4\rangle, \langle1, 2, 4\rangle, \langle2, 3, 6\rangle, \langle1, 3, 6\rangle, \langle1, 2, 6\rangle\} \\
  \Si_2 &=& \{\langle2, 4, 5\rangle, \langle1, 4, 5\rangle, \langle1, 3, 5\rangle, \langle3, 5, 6\rangle, \langle2, 5, 6\rangle, \langle1, 2, 4\rangle, \langle1, 3, 6\rangle, \langle1, 2, 6\rangle\} \\
  \Si_3 &=& \{\langle2, 4, 5\rangle, \langle1, 4, 5\rangle, \langle1, 3, 5\rangle, \langle3, 5, 6\rangle, \langle2, 5, 6\rangle, \langle2, 4, 6\rangle, \langle1, 4, 6\rangle, \langle1, 3, 6\rangle\} \\
  \Si_4 &=& \{\langle2, 4, 5\rangle, \langle2, 3, 5\rangle, \langle1, 4, 5\rangle, \langle1, 3, 5\rangle, \langle1, 2, 4\rangle, \langle2, 3, 6\rangle, \langle1, 3, 6\rangle, \langle1, 2, 6\rangle\} \\
  \Si_5 &=& \{\langle3, 4, 5\rangle, \langle2, 4, 5\rangle, \langle3, 5, 6\rangle, \langle2, 5, 6\rangle, \langle1, 3, 4\rangle, \langle2, 4, 6\rangle, \langle1, 4, 6\rangle, \langle1, 3, 6\rangle\} \\
  \Si_6 &=& \{\langle3, 4, 5\rangle, \langle2, 4, 5\rangle, \langle2, 3, 5\rangle, \langle1, 3, 4\rangle, \langle2, 4, 6\rangle, \langle2, 3, 6\rangle, \langle1, 4, 6\rangle, \langle1, 3, 6\rangle\} \\
  \Si_7 &=& \{\langle3, 4, 5\rangle, \langle2, 4, 5\rangle, \langle3, 5, 6\rangle, \langle2, 5, 6\rangle, \langle1, 3, 4\rangle, \langle1, 2, 4\rangle, \langle1, 3, 6\rangle, \langle1, 2, 6\rangle\} \\
  \Si_8 &=& \{\langle2, 4, 5\rangle, \langle2, 3, 5\rangle, \langle1, 4, 5\rangle, \langle1, 3, 5\rangle, \langle2, 4, 6\rangle, \langle2, 3, 6\rangle, \langle1, 4, 6\rangle, \langle1, 3, 6\rangle\}
\end{eqnarray*}
\begin{figure}
\begin{center}
\includegraphics[width=7truecm]{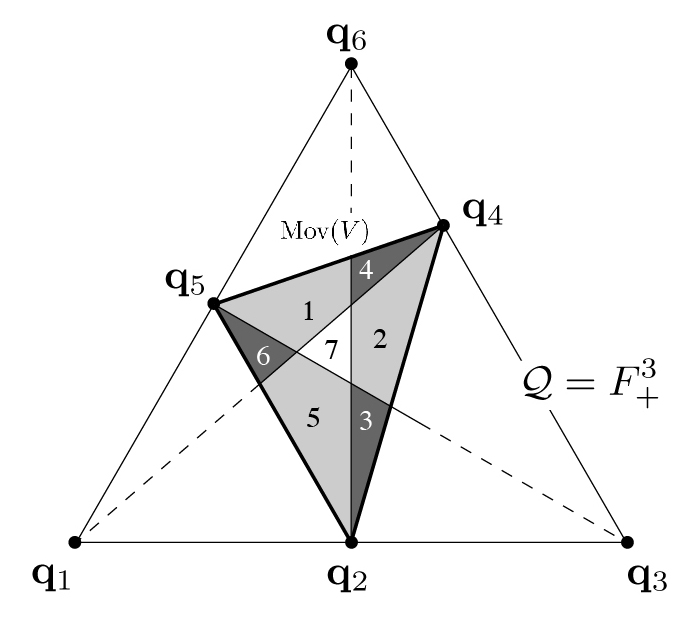}
\caption{\label{Fig3}The section of the effective divisors cone $\gkz=F^3_+$ with the plane $x_1+x_2+x_3=1$.}
\end{center}
\end{figure}
where $\langle i,j,k\rangle$ denotes the cone spanned in $N_{\R}$ by $i$--th, $j$--th and $k$--the columns of the fan $V$. Fans $\Si_i$, for $1\leq i\leq 7$, give rise to projective toric varieties, each of them being associated with a full dimensional chamber inside the Moving cone $\Mov(V)\subseteq \gkz$ of $X$, which actually depends on the 1--skeleton of the fan, only, i.e. on the columns of $V$: the situation is represented in Fig. \ref{Fig3}, by cutting out the positive orthant $F^3_+=\gkz$, inside $F^3_{\R}$, with the hyperplane $x_1+x_2+x_3=1$. On the contrary the last fan $\Si_8$ does not give rise to a projective toric variety since its Gale dual bunch of cones is given by
\begin{equation*}
  \B_{\Si_8}=\{\langle1, 3, 6\rangle, \langle1, 4, 6\rangle, \langle2, 3, 6\rangle, \langle2, 4, 6\rangle, \langle1, 3, 5\rangle, \langle1, 4, 5\rangle, \langle2, 3, 5\rangle, \langle2, 4, 5\rangle\}
\end{equation*}
where now $\langle i,j,k\rangle$ denotes the cone spanned in $F_{\R}^3$ by $i$--th, $j$--th and $k$--the columns of the weight matrix $Q$. One can easily check that the intersection of all the cones in the bunch $\B_{\Si_8}$ is trivial, hence giving
\begin{equation*}
  \Nef(X(\Si_8))=\bigcap_{\b\in\B_{\Si_8}} \b= 0\,.
\end{equation*}
For sake of completeness, let us finally observe that, using technics like those given in \cite[Thm.~4]{RT-Ample}, one can check that the anti-canonical class lies on the wall separating chamber 2 and 7 in Fig.~\ref{Fig3}: in particular none of the fans in $\SF(V)$ give a Gorenstien toric varieties but $X(\Si_2)$ and $X(\Si_7)$ are \emph{weak $\Q$-Fano} toric varieties.

\begin{remark}\label{rem:} The previous example has been obtained by \emph{deforming} the Berchtold--Hausen example \cite[Ex.\,10.2]{Berchtold-Hausen} in the following sense. The latter can be obtained as above, by starting from the weight matrix
\begin{equation*}
Q_{BH}:=\left(
     \begin{array}{cccccc}
       1&1&0&0&1&0\\
      0&1&1&1&0&0\\
       0&0&0&1&1&1 \\
     \end{array}
   \right)\
\ \Rightarrow\ V_{BH}= \left(
  \begin{array}{cccccc}
    1&0&0&0&-1&1 \\
    0&1&0&-1&-1&2 \\
    0&0&1&-1&0&1 \\
  \end{array}
\right)
\end{equation*}
\begin{figure}
\begin{center}
\includegraphics[width=8truecm]{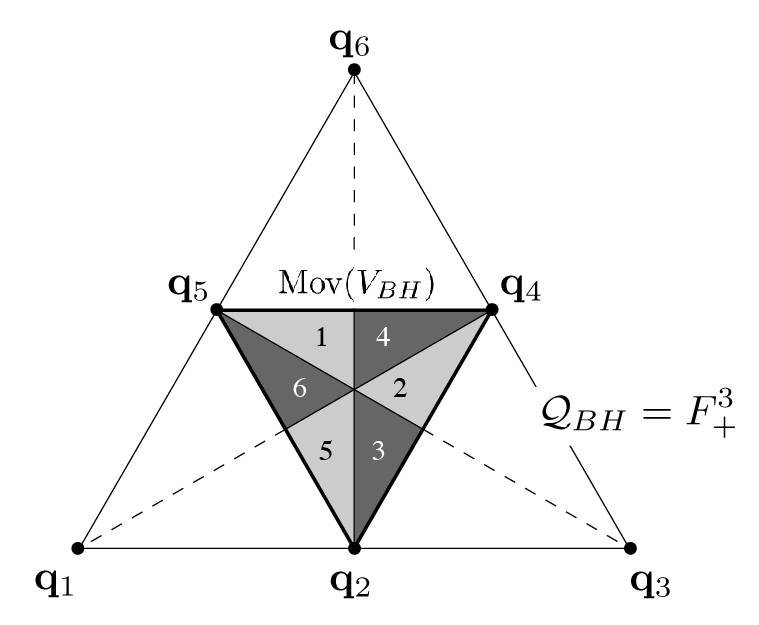}
\caption{\label{Fig4}The section of the effective divisors cone $\gkz_{BH}=F^3_+$, in \cite[Ex.\,10.2]{Berchtold-Hausen}, with the plane $x_1+x_2+x_3=1$.}
\end{center}
\end{figure}
Also in this case we obtain 8 rational, simplicial and complete fans whose fan matrix is $V_{BH}$, given by the $2^3$ possible subdivision of the three quadrangular faces of the prism whose vertexes are given by the columns of $V_{BH}$ (see Fig.\,\ref{Fig5}).  But only 6 of them give rise to projective varieties, while the remaining 2 fans originates two distinct complete and non--projective $\Q$--factorial 3--dimensional varieties of Picard number 3. The situation is represented in  Fig.\,\ref{Fig5} and Gale dually in Fig. \ref{Fig4}. Explicitly the two non--projective fans are the following:
\begin{eqnarray*}
             \Si&=& \{\langle2, 4, 5\rangle, \langle2, 3, 5\rangle, \langle1, 4, 5\rangle, \langle1, 3, 5\rangle, \langle2, 4, 6\rangle, \langle2, 3, 6\rangle, \langle1, 4, 6\rangle, \langle1, 3, 6\rangle\} \\
            \Si' &=& \{\langle3, 4, 5\rangle, \langle2, 4, 5\rangle, \langle3, 5, 6\rangle, \langle2, 5, 6\rangle, \langle1, 3, 4\rangle, \langle1, 2, 4\rangle, \langle1, 3, 6\rangle, \langle1, 2, 6\rangle\}
\end{eqnarray*}
whose Gale dual bunches are given by
\begin{eqnarray*}
  \B &=& \{\langle1, 3, 6\rangle, \langle1, 4, 6\rangle, \langle2, 3, 6\rangle, \langle2, 4, 6\rangle, \langle1, 3, 5\rangle, \langle1, 4, 5\rangle, \langle2, 3, 5\rangle, \langle2, 4, 5\rangle\} \\
  \B' &=& \{\langle1, 2, 6\rangle, \langle1, 3, 6\rangle, \langle1, 2, 4\rangle, \langle1, 3, 4\rangle, \langle2, 5, 6\rangle, \langle3, 5, 6\rangle, \langle2, 4, 5\rangle, \langle3, 4, 5\rangle\}\,.
\end{eqnarray*}
Looking at Fig. \ref{Fig4}, one can easily check that
\begin{equation*}
  \Nef(X)=\bigcap_{\b\in\B} \b= \left\langle\begin{array}{c}
                                        1 \\
                                        1 \\
                                        1
                                      \end{array}
  \right\rangle=\bigcap_{\b'\in\B'} \b'=\Nef(X')\,,
\end{equation*}
giving the 1--dimensional cone generated by the class $\left\langle\begin{array}{c}
                                        3 \\
                                        3 \\
                                        3
                                      \end{array}
  \right\rangle$ of the anti--canonical divisor.

  \noindent Fig.~\ref{Fig4} explains as the anti-canonical class lies on the intersection of all the chambers, so giving that every fan in $\SF(V_{BH})$ gives a non-Gorestein weak $\Q$-Fano toric variety.

   The two fans $\Si$ and $\Si'$ corresponds to symmetric subdivisions of the three quadrangular faces of the prism ge\-ne\-ra\-ted by the columns of $V$ (see Fig.\,\ref{Fig5}). Such a symmetry is Gale dually reflected by the position of the anti--canonical class with respect the subdivision of $\Mov(V_{BH})$ giving the secondary fan. More precisely, Fig.\,\ref{Fig4} clearly shows that $\Mov(V_{BH})$  and the secondary fan turn out to be invariant under the action of the three symmetries with respect to the three axis of the triangle representing the section of $\gkz=F^3_+$ with the plane $x_1+x_2+x_3=1$, i.e. the symmetries of $F^3_{\R}$ with respect to the three planes supporting the 2--dimensional cones $\langle 1,4\rangle$, $\langle 2,6\rangle$ and $\langle 3,5\rangle$. These symmetries correspond to permutations on the columns of the weight matrix $Q$ given by $(2\,5)(3\,6)$, $(1\,3)(4\,5)$ and $(1\,6)(2\,4)$, respectively. Notice that bunches $\B$ and $\B'$ are related each other by every one of those permutations. Geometrically this means that $X$ and $X'$ are related by a bi--meromorphic map, still called a \emph{flip}, given by the contraction of three facets of the fan, namely the closures of torus orbits $\overline{O(\langle 2,3\rangle)}$, $\overline{O(\langle 1,5\rangle)}$ and $\overline{O(\langle 4,6\rangle)}$, followed by the small resolution whose exceptional locus is given by the union of the symmetric facets   $\overline{O(\langle 5,6\rangle)}$, $\overline{O(\langle 3,4\rangle)}$ and $\overline{O(\langle 1,2\rangle)}$ (compare with Fig.\,\ref{Fig5}). Then $X$ and $X'$ are isomorphic in codimension 1, but not in codimension 2.

  If we now consider the weight matrix
  \begin{equation*}
Q_{t}:=\left(
     \begin{array}{cccccc}
       1&1&0&0&1&0\\
      0&1&1&1&0&0\\
       0&0&0&t&1&1 \\
     \end{array}
   \right)\
\ \Rightarrow\ V_{t}= \left(
  \begin{array}{cccccc}
    1&0&0&0&-1&1 \\
    0&1&0&-1&-1&1+t \\
    0&0&1&-1&0&t \\
  \end{array}
\right)
\end{equation*}
with $t$ a positive integer number, we get the Berchtold--Hausen example for $t=1$, our example for $t=2$ and an infinite number of distinct examples analogous to the latter,  for $t>1$. If we look at the evolution of the fan $\Si_7$, hence of the chamber 7 in Fig. \ref{Fig3}, we get a countable family $\mathcal{X}_7\to \N\setminus\{0\}$ of $\Q$--factorial projective toric varieties degenerating to a non--projective complete one for $t=1$.

\begin{figure}
\begin{center}
\includegraphics[width=13truecm]{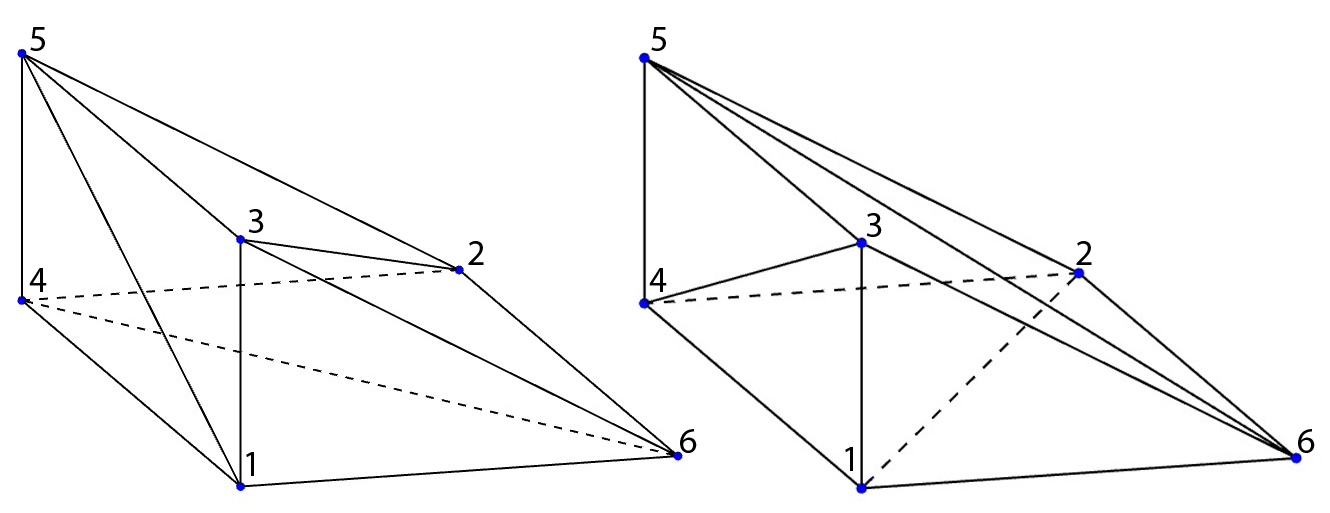}
\caption{\label{Fig5}The two symmetric subdivisions of the prism, whose vertexes are given by the columns of $V$, and giving rise to two distinct and bi--meromorphic $\Q$--factorial complete and non--projective toric varieties $X(\Si)$ and $X(\Si')$, respectively. The former is that given by \cite[Ex.\,10.2]{Berchtold-Hausen}. Notice that all the possible fans in $\SF(V)$ are given by the $2^3$ possible subdivisions of the quadrangular faces.}
\end{center}
\end{figure}

Let us finally observe that a similar situation, although with Picard number 4, can be observed for the Oda's example \cite[p.\,84]{Oda} too, from which the Berchtold--Hausen one is obtained by suppressing a fan generator. In Oda's notation, the symmetric non--projective fan, with respect to the one presented in \cite{Oda}, can be obtained by exchanging $n_2\leftrightarrow n_3$ and $n_2'\leftrightarrow n_3'$.
\end{remark}

 \begin{acknowledgements} We would like to thank Cinzia Casagrande for helpful conversations and suggestions. We also thank Daniela Caldini for her contribution in making the figures of the present paper.
\end{acknowledgements}

\end{document}